\documentclass[12pt]{amsart}

\usepackage{amsmath}
\usepackage{amssymb}
\usepackage{fancybox}
\usepackage{eucal}
\usepackage{amsthm}
\usepackage{amscd}
\usepackage{hyperref}
\usepackage{wasysym}
\usepackage{url}

\usepackage{amssymb,}
\usepackage[]{amsmath, amsthm, amsfonts,graphicx, amscd,}
\usepackage[all,cmtip]{xy}
\input amssym.def \input amssym

\usepackage[]{amsmath, amsthm, amsfonts,graphicx, amscd, cite, }
\usepackage[all,cmtip]{xy}
\input amssym.def \input amssym
\begin{document}

\newtheorem {thm}{Theorem}[section]
\newtheorem{corr}[thm]{Corollary}
\newtheorem {alg}[thm]{Algorithm}
\newtheorem*{thmstar}{Theorem}
\newtheorem{prop}[thm]{Proposition}
\newtheorem*{propstar}{Proposition}
\newtheorem {lem}[thm]{Lemma}
\newtheorem*{lemstar}{Lemma}
\newtheorem{conj}[thm]{Conjecture}
\newtheorem*{conjstar}{Conjecture}
\theoremstyle{remark}
\newtheorem{rem}[thm]{Remark}
\newtheorem{np*}{Non-Proof}
\newtheorem*{remstar}{Remark}
\theoremstyle{definition}
\newtheorem{defn}[thm]{Definition}
\newtheorem*{defnstar}{Definition}
\newtheorem{exam}[thm]{Example}
\newtheorem{ques}[thm]{Question}
\newcommand{\pd}[2]{\frac{\partial #1}{\partial #2}}
\newcommand{\pdtwo}[2]{\frac{\partial^2 #1}{\partial #2^2}}
\def\Ind{\setbox0=\hbox{$x$}\kern\wd0\hbox to 0pt{\hss$\mid$\hss} \lower.9\ht0\hbox to 0pt{\hss$\smile$\hss}\kern\wd0}
\def\Notind{\setbox0=\hbox{$x$}\kern\wd0\hbox to 0pt{\mathchardef \nn=12854\hss$\nn$\kern1.4\wd0\hss}\hbox to 0pt{\hss$\mid$\hss}\lower.9\ht0 \hbox to 0pt{\hss$\smile$\hss}\kern\wd0}
\def\ind{\mathop{\mathpalette\Ind{}}}
\def\nind{\mathop{\mathpalette\Notind{}}} 
\newcommand{\m}{\mathbb }
\newcommand{\mc}{\mathcal }
\newcommand{\mf}{\mathfrak }
\title{Isogeny in Superstable Groups}
\author{James Freitag}

\address{Department of Mathematics, Statistics, and Computer Science\\
University of Illinois at Chicago\\
322 Science and Engineering Offices (M/C 249)\\
851 S. Morgan Street\\
Chicago, IL 60607-7045 }
\email{freitagj@gmail.com}
\maketitle
\begin{abstract}  We study and develop a notion of isogeny for superstable groups inspired by work of \cite{BerlineLascar}, \cite{CassidySinger} and \cite{Baudisch}. We prove several fundamental properties of the notion and then use it to formulate and prove uniqueness results. Connections to existing model theoretic notions are explained. 
\end{abstract}

The main goal of this note is to develop notions of \emph{strong connectedness}, \emph{almost simplicity} and \emph{isogeny} for the class of superstable groups, in analogy to the related notions for algebraic groups. Although we will not discuss \emph{supersimple} groups here, the results should can be generalized to that setting. In this paper, notions like simple, quasi-simple, and almost simple are group theoretic notions and have nothing to do with the similarly named model theoretic property of first order theories. We will then use the notion of isogeny to prove results of the form ``the construction is unique up to isogeny". For an example from algebraic groups, see the Jordan-H\"{o}lder theorem. The other guiding example will be the Cassidy-Singer analysis of differential algebraic groups. 

The methods here specialize to the known results in algebraic groups. The algebraic groups case also inspired the work of Cassidy and Singer in the differential setting. Both the algebraic and differential algebraic approaches inspire the work here.  Many of the proofs in this paper are translations of proofs from these cases, generalized and modified appropriately. One interesting note is that while $U$-rank specializes to (Krull) dimension in algebraic groups, the notion of dimension that Cassidy and Singer use in their analysis is not $U$-rank in differentially closed fields. Cassidy and Singer use the gauge of the differential algebraic group, that is the pair $(a_{\tau}, \tau),$ where $a_ \tau$ is the typical differential dimension and $\tau$ is the differential type. From these differential birational invariants, one can formulate an upper bound for Lascar rank in differential fields. There is no known lower bound for Lascar rank in terms of these invariants \cite{Suerthesis}. We will define similar notation in the superstable setting.

Strong results on the structure of infinite rank superstable groups were first established in \cite{BerlineLascar}. Further model theoretic analysis continued over the next several years and is recalled in \cite{Poizat}. Our purpose here is somewhat different from the existing model theoretic analysis. The basic notion we consider is \emph{isogeny}. The notion is interesting in its own right, and we prove several results about the properties of the condition. We hope to illustrate how to import techniques from differential algebraic groups into superstable groups, even when (as in this case) the results are not necessarily generalizations. This translation goes via thinking about Lascar rank in the way that differential algebraists think about the gauge of a differential algebraic group. Further, we hope this will lead to future work in model theory of fields with more general operators in which Lascar rank is either difficult to understand and calculate or is simply not available. The decomposition theorem proved here is close to the one proved by Baudisch \cite{Baudisch}. The quotients in our decomposition are almost simple and might have infinite centers; our decomposition is coarser than Baudisch's decomposition. Baudisch's paper does not mention the issue of uniqueness of the decomposition. The style and techniques for proving the decomposition theorem in this paper follow proofs of theorems from algebraic groups much more closely than then development contained in \cite{Baudisch}.

The work in this paper has been obviously influenced by that of Phyllis Cassidy and Michael Singer \cite{CassidySinger}. The former brought their work to the author's attention during a trip to the Kolchin seminar at CUNY. The author would also like to thank Dave Marker and John Baldwin for enlightening discussions. 

\section{Notation and Preliminaries}
Throughout this note, $G$ is a group definable in a superstable theory $T.$ We will heavily use the notion of \emph{Lascar rank} on types, denoted $RU(p).$ Though this is a rank on types, one can abuse notation and denote, by $RU(G)=RU(p_G)$, where $p_G$ is a generic type of $G.$ For certain technical reasons, this might be somewhat problematic when dealing with arbitrary definable sets, but not when dealing with (type-) definable groups. For this paper, we will assume that $\alpha$ and $\beta$ are ordinals such that $RU(G)=\omega^\alpha \cdot n + \beta$ where $\beta < \omega^\alpha$ (note that this is no restriction at all on the group $G$). Lascar rank ($U$-rank, $RU$) is the main tool used in this paper, and properly it is a rank on types. We abuse notation in a standard way and write $RU(X),$ where $X$ is a definable set (usually a group, in fact). In this case, the Lascar rank of the set is the supremum of the Lascar ranks of the complete types which include the formula $``x \in X.$

A group is called type-definable if it is an intersection of definable subgroups. We will be assuming standard notation from superstable group theory except where we define new notation. Poizat's Stable Groups \cite{Poizat} is suggested as a reference for the notation which is not explicitely defined. The reader is advised that we will make frequent use of the Lascar inequality in particular. We emphasize that we are working in some fixed superstable theory $T$, and are calculating Lascar rank within that theory. 
 
\begin{defn} Define $\tau_U: \{Def(G)\} \rightarrow On$ to be the highest power $\alpha,$ such that $\omega^\alpha$ appears in the Cantor normal form of the Lascar rank of definable set in question. $G$ is \emph{$\alpha$-connected} if for every type-definable normal subgroup $H$ of $G,$ $\tau_U(G/H)= \tau_U(G).$ 

We will also call $\alpha$-connected groups \emph{strongly connected}.  
$G$ is \emph{almost simple} if there is no type definable subgroup $H$ of $G,$ with $\tau _U (G/H)< \tau _U(G).$ 
\end{defn}
\begin{rem} This notation is not standard, but it is convenient for the purposes here. It is inspired by the notation of \cite{CassidySinger}. The definition of $\alpha$-connected agrees with that of \cite{BerlineLascar}. The following open question depends on the relationship between Lascar rank and gauge in differential algebraic groups: 
\end{rem}
\begin{ques} Is a strongly connected differential algebraic group (strongly connected in the sense of differential gauge) actually strongly connected in the sense of Lascar rank? 
\end{ques}

It is known, by results of Berline and Lascar \cite{BerlineLascar}, that $G$ is $\alpha$-connected if and only if $RU(G)=\omega^\alpha \cdot n$ and $G$ is connected (in the traditional sense that there is no type-definable subgroup of finite index).  This follows after they find normal subgroups of rank $\omega^\alpha \cdot n$ (where $\tau _U(G)=\alpha$ and $n$ is the coefficient of $\omega^\alpha$ in the Cantor normal for of the Lascar rank of $G$), which means that the quotient by such a subgroup has lower leading monomial for its $U$-rank, by the Lascar inequality. Connectedness is also an obvious necessary condition. The sufficiency of the above conditions of Lascar and Berline also follows by application the Lascar inequality.

 So, we are limited to groups of monomial valued $U$-rank. In that case, being $\alpha$-connected is equivalent to being connected (ie, no finite index definable subgroups).   For an arbitrary superstable group $G,$ let the set $\mc S= \{H \subset G \, | \, RU(G/H) < \omega^\alpha \}=\{H \subset G \, | \, \tau_U(G/H) < \tau_U(G) \}.$ The appropriate notion of simple, which we call \emph{almost simple} is the condition that $\tau_U(H) < \tau_U(G)$ for all definable $H$ in $G.$ Note here that $H$ need not be normal to consider $G/H$ as a left coset space.

\begin{prop}\label{central} Suppose that $G$ is $\alpha$-connected. Every type-definable normal subgroup, $N,$  with $\tau_U(N) < \tau_U(G)$ is central. 
\end{prop}
\begin{proof} Consider the map $\alpha: G \times N \rightarrow N$ given by $(g,a) \mapsto gag^{-1}.$  For any fixed $a \in N,$ $\alpha_a(g):=gag^{-1}$ is a definable map from $G$ to $N,$ such that $\alpha_a$ is constant on left cosets of the centralizer of $a$, $Z_G(a).$ So, there is a definable map $\beta,$ such that the diagram commutes, 

$$\xymatrix{G \ar[d]_\pi \ar[r]^{\alpha_a}&
N \\
G/Z_G(a) \ar@{.>}[ru]_\beta}$$
We note that $\alpha_a(g)=\alpha_a(h)$ implies that $h^{-1}g \in Z_G(a).$ Thus, $\beta$ is injective. But, then $\tau_U(G/Z_G(a)) \leq \tau_U(N) < \tau(G),$ so $Z_G(a)$ must be all of $G$, since otherwise we have found a subgroup such that the $U$-rank of the coset space has leading monomial in its Cantor normal form less than $\tau_U(G).$ This means that the rank of $Z_G(a)$ is at least equal to the leading monomial. On the face of things, this should not force $Z_G(a)$ to be all of $G,$ since we do not know that $Z_G(a)$ is a normal subgroup of $G.$ But, in general, one now knows that the set of subgroups $H$ of $G$ such that the coset space has rank less than $\omega^\alpha$ is nonempty. This set is closed under intersections and the minimal element in the set will be a definable characteristic (so normal) subgroup of $G$ which shows that $G$ is not $\alpha$-connected. So, it must be that $G=Z_G(a).$ 
\end{proof}

To make clear the last several lines of the proof, here is an easy lemma which has appeared in many places. 
\begin{lem} Suppose that $H_1$ and $H_2$ are definable subgroups such that the coset spaces $G/H_1$ and $G/H_2$ have U-rank less than $\omega^\alpha$. Then the same is true of $H_1 \cap H_2.$  
\end{lem}
\begin{proof} There is a natural injection of coset spaces $H_1/H_1 \cap H_2 \rightarrow G/H_2,$ so $$RU(H_1/H_1 \cap H_2) < RU(G/H_2).$$ Now, by the Lascar inequality, $RU(G/H_1 \cap H_2) \leq RU(G/H_1) \oplus RU(H_1/H_1 \cap H_2) .$ In particular, since both of the terms of the sum have U-rank less than $\omega^\alpha,$ so does the sum. Thus, the set $\mc S$ is closed under intersection. We further note that the minimal element in this set is a characteristic subgroup and is called the $\alpha$-connected component.
\end{proof}

\begin{prop}\label{imageconnected} The image of an strongly connected group under a definable homomorphism is strongly connected or trivial. 
\end{prop}

\begin{proof} Suppose that the image is nontrivial and not strongly connected. Then taking the inverse image of the definable subgroup of the image which shows non-strong connectedness would show the non-strong connectedness of $G$ itself. 
\end{proof}

\section{Isogeny}
The notion of strongly connected (or $\alpha$-connected - recall we assume that the Lascar rank of $G$ has $\omega ^\alpha$ appearing as the leading term in its Cantor normal form) plays the role that connected plays in algebraic groups. Almost simple plays the role of quasi simple. Now we define isogeny in this setting. 

\begin{defn} Suppose that $G$ and $H$ are $\alpha$-connected. Then a group homomorphism $\phi : G \rightarrow H$ is an isogeny if $\phi$ is surjective and $\tau_U (Ker \phi) < \tau_U (G).$ We way that $H_1$ and $H_2$ are isogenous if there are $\phi_i: G \rightarrow H_i$ which are isogenies. 
\end{defn}

We are not generally dealing with definability problems in this paper, so even if we do not explicitly say so, groups are assumed to be type-definable.

\begin{prop} \label{isogenyfiber}
Let $G_1$ and $G_2$ be $\alpha$-connected subgroups. The following are equivalent: 
\begin{itemize}
\item There is an $\alpha$-connected group $H$ and isogenies $\phi_i: H \rightarrow G_i$:\\ 
$$\xymatrix{
&H \ar[rd]^{\phi_2}  \ar[ld]_{\phi_1} \\
G_1  &
 &
G_2 
}$$

\item There is an $\alpha$-connected group $K$ and isogenies $\psi_i: G_i \rightarrow K:$\\
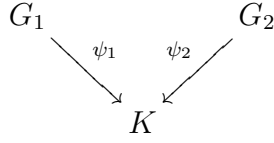
$$\xymatrix{
G_1 \ar[rd]^{\psi_1}  &
 &
G_2 \ar[ld]_{\psi_2} \\
&K 
}$$
\end{itemize}
\end{prop}
\begin{proof} 
Let $H$ and $\phi_i$ be as in condition 1). Let $H_1 = \phi_1 (ker \phi_2)$ and $H_2 = \phi_2(ker \phi_1).$ Then $H_1 = \phi_1 (ker \phi_1 ker \phi_2),$ and $H_2 = \phi_2 (ker \phi_1 ker \phi_2).$ Then $$G_1 /H_1 = \phi_1(H)/ \phi_1(ker \phi_2)=\phi(H)/ \phi_1(ker \phi_1 ker \phi_2)=H/ (ker\phi_1 \phi_2).$$
$$G_2/H_2 = \phi_2(H)/\phi_2(ker(\phi_2))=\phi_2(H)/\phi_2(ker \phi_1 ker \phi_2)=H/(ker \phi_1 ker \phi_2).$$

So, let $K=H/(ker \phi_1 ker \phi_2).$ $K,$ being the image of an $\alpha$-connected group $H$ is $\alpha$-connected. Further, $\tau_U (ker \phi_i)<\alpha,$ so $\tau_U(ker \phi_1 ker \phi_2)< \alpha.$ But, then letting $\psi_i$ be the projection map $G_i \rightarrow G_i/H_i =K.$ We have shown that $\psi_i$ is an isogeny. 

Now, assume condition 2). We let $G= \{ (g_1 ,g_2) \in G_1 \times G_2 \, | \, \psi_1(g_1)=\psi_2(g_2) \}.$ Then there are natural surjective projections $\phi_i:G \rightarrow G_i.$ But, then we see that $\tau_U(G) \geq \tau_U(G_i).$ As the kernel of the projection maps, $\phi_i,$ are contained in $ker \psi_1 \times ker \psi_2,$ the Lascar rank of the kernels of the maps is less than $\omega^ \alpha,$ since both of the groups in the product are (by virtue of $\psi_i$ being an isogeny). So, $\phi_i$ is an isogeny. 
\end{proof}

\begin{prop} Isogeny is an equivalence relation on the $\alpha$-connected type-definable subgroups in a superstable group $G$ with $RU(G) = \omega^ \alpha + \beta.$ Let $G_1$ and $G_2$ be isogenous. Then,\begin{itemize}
\item There is a bijection, $r$ between the type-definable subgroups $G_1 \leq G$ with $\tau_U(G_1)=\tau_U(G)$ and those $K_1 \leq K$ with $\tau_U(K_1)=\tau_U(K).$ 
\item Suppose that $r(G_1)=K_1$ and $r(G_2)=K_2.$ \\ $G_1 \leq G_2$ if and only if $K_1 \leq K_2.$ \\ $G_1 \lhd G_2$ if and only if $K_1 \lhd K_2.$ 
\item If, as above, $G_1 \lhd G_2,$ then $\tau_U (G_2/G_1)= \alpha,$ $G_2 /G_1$ is strongly connected, and $G_2/ G_1 $ is isogenous to $K_2 /K_1.$
\item Products of isogenous groups are isogenous. 
\end{itemize}
\end{prop}
\begin{proof} Reflexivity and symmetry of the isogeny relation are clear. Now, Suppose that $H_1$ is isogenous to $H_2$ and $H_2$ is isogenous to $H_3.$ Then, we have a diagram of isogenies with $\alpha$-connected $K_1$ and $K_2$: 

$$
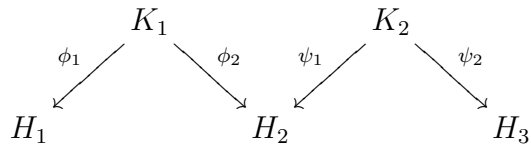
\xymatrix{
& K_1 \ar[rd]^{\phi_2} \ar[ld]_{\phi_1} &  & K_2 \ar[rd]^{\psi_2} \ar[ld]_{\psi_1} & \\
H_1 & &  H_2 & & H_3 
}$$

But, by \ref{isogenyfiber}, we get the following diagram, with isogenies and $\alpha$-connected $L:$

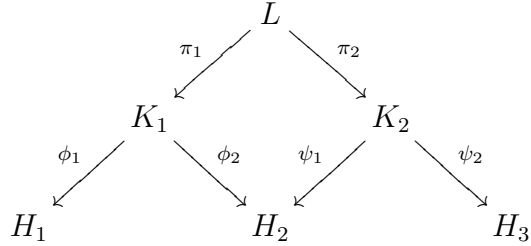
$$\xymatrix{
& &  L \ar[rd]^{\pi_2} \ar[ld]_{\pi_1} &  & \\
& K_1 \ar[rd]^{\phi_2} \ar[ld]_{\phi_1} &  & K_2 \ar[rd]^{\psi_2} \ar[ld]_{\psi_1} & \\
H_1 & &  H_2 & & H_3 
}$$

For $H_1$ to be isogenous to $H_3$, we would require that $\phi_1 \circ \pi_1$ and $\psi_2 \circ \pi_2$ are isogenies. Surjectivity is obvious. To show that the kernel of either of the compositions is of $U$-rank less than $\omega^\alpha.$ The fiber of $\pi_1$ over any point of $K_1.$ is a coset of the kernel of $\pi_1.$ Therefore, by the Lascar inequality, the kernel of the map $\phi_1 \circ \pi_1$ is bounded above by $RU(a) \oplus RU(ker \pi),$ where $a$ is an element of the kernel of $\phi_1.$ Of course, this implies that $RU(a) < \omega^\alpha.$ So, $\tau_U(ker(\phi_1 \circ \pi_1))<\alpha.$ Then, by a symmetric argument on $\psi_2 \circ \pi_2,$ both maps are isogenies. \\
Suppose that we have the following diagram: 

$$
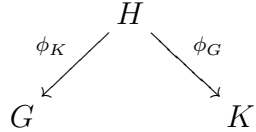
\xymatrix{
& H \ar[rd]^{\phi_G} \ar[ld]_{\phi_K}    \\
G & &  K 
}$$
Then we claim there is a bijection between the sets of subgroups of $G_1 \leq G$ and $K_1 \leq K$ with $\tau_U(G_1)=\tau_U(G)=\tau_U(K_1)=\tau_U(K).$ We will set up a correspondence between these two types of subgroups. Let $r(G_1) = K_1 $ if there is a definable $\alpha$-connected subgroup $H_1 \leq H$ with $\phi_G(H_1)=G_1.$ To show that the map $r$ is well-defined and bijective, it suffices to show that there is a unique choice of $\alpha$-connected subgroup $H_1 \leq H$ with $\phi_G(J)=H_1.$ Of course, there is one natural candidate, namely, the $\alpha$-connected component of the inverse image of $G_1,$ which we will denote $\phi_G^{-1} (G_1)^{(\alpha)}.$  Certainly, by the Lascar inequality and the fact that $\tau_U(ker \phi_G)< \alpha,$ we know that $\tau(\phi_G^{-1} (G_1))=\alpha.$ So, at least $\phi_G^{-1} (G_1)^{(\alpha)}$ is a definable group which is $\alpha$-connected and of suitable rank. We claim that $\phi_G (\phi_G^{-1} (G_1)^{(\alpha)})=G_1.$ Of course, the image is contained in $G_1.$ But, suppose that $RU(G_1)=\omega^\alpha \cdot n + \beta.$ So, $RU( \phi^{-1}(G_1))=\omega^\alpha n + \gamma.$ That the Lascar rank of the inverse image is at least this big for some small $\gamma$ is trivial. That it is at most this big follows from the Lascar inequality and the fact that $RU(ker \phi_G \cap \phi^{-1}(G)) \leq RU(ker \phi_G) < \omega ^\alpha.$ So, the image of $\phi^{-1}_G(G_1)^{(\alpha)}$ is a strongly connected subgroup of $G_1$ of the same leading monomial $U$-rank. But, then, by Berline-Lascar, the image is $G_1.$ Now, we claim that there is no other choice of $H_1.$ If there was, it would have to be a proper definable subgroup of $\phi^{-1}_G(G_1)^{(\alpha)}$. But, we know that all such subgroups have leading monomial $U$-rank less than $\phi^{-1}_G(G_1)^{(\alpha)}$ by virtue of $\alpha$-connectedness (iff monomial valued $U$-rank and $\alpha$-connectedness). Of course, then the image of such a group can not be all of $G_1,$ simply by virtue of rank. The correspondence is bijective, since the image of a $\alpha$-connected subgroup $H_1$ under $\phi_G$ is an $\alpha$-connected \ref{imageconnected} subgroup $G_1$ of $G$ with the same $U$-rank (by the now familiar "rank of the kernel is small" argument). Thus, there is a bijective correspondence between the $\alpha$-connected subgroups of $H$ and $G.$ This argument is, completely symmetric, so there is also such a correspondence for $H$ and $K.$\\
All of the subgroups in the following paragraph are $\alpha$-connected. Suppose that $r(G_1)=K_1$ and $r(G_2)=K_2.$ Then suppose that $G_1 \leq G_2.$ Then $\phi^{-1}_G(G_1)^{(\alpha)} \leq \phi^{-1}_G(G_2)^{(\alpha)},$ because $\alpha$-connected subgroups of $\phi^{-1}_G(G_2)$ must be contained in the $\alpha$-connected component. 

Of course, this implies that $K_1=\phi_K(\phi^{-1}_G(G_1)^{(\alpha)}) \leq \phi_K(\phi^{-1}_G(G_2)^{(\alpha)})=K_2.$
Now we assume that $G_1 \lhd G_2$. Then $\phi^{-1}_G (G_1) \lhd \phi^{-1}_G (G_2).$  Since the $\alpha$-connected component of a group is characteristic, $\phi^{-1}_G (G_1)^{(\alpha ) } \lhd \phi^{-1}_G (G_2).$ So, $\phi^{-1}_G (G_1)^{(\alpha ) } \lhd \phi^{-1}_G (G_2)^{(\alpha ) }.$ But, then $K_1 = \phi_K(\phi^{-1}_G (G_1)^{(\alpha ) }) \lhd \phi_K (\phi^{-1}_G (G_2)^{(\alpha ) }=K_2.$ \\ The maps induced by $\phi_G$ and $\phi_K$ on the quotient $H_2 /H_1$ is an isogeny, since it is surjective onto its image and the kernel of the map is the quotient of the kernel of an isogeny and $\tau(H_2 /H_1)= \tau(H).$ \\
Products of isogenous groups are isogenous, becuase taking a product of the isogeny maps gives an isogeny map (surjectivity is clear and the $U$-rank of the kernel is bounded by the Cantor sum of the $U$-rank of the kernels in the product). 
\end{proof}

\begin{rem} For more details on the following brief remarks, see \cite{Poizat}. In superstable theories, all types are coordinatized by \emph{regular} types. One often considers the equivalence relation of \emph{nonorthogonality} of the regular types. The strongly connected groups considered here have generics which are a product of regular types, each nonorthogonal to a type of rank $\omega ^\alpha.$ The equivalence relation of nonorthogonality is much coarser than isogeny. The isogeny relation on almost simple groups is finer, and takes into account the group theoretic properties of the definable group in ways which nonorthogonality does not. 

Let $G$ be a (non-commutative) quasi-simple algebraic group. In algebraically closed fields, the nonorthogonality relation is trivial, since any two positive rank types are nonorthogonal. The isogeny relation is nontrivial, and it matches the classical definition. Even in settings in which the nonorthogonality relation is highly nontrivial (for instance differentially closed fields), the isogeny relation is finer. Of course, almost simplicity is not a sufficient condition for a connected group to have regular generic type. In the setting of differential algebraic groups, is it necessary?
\end{rem}

\begin{lem}\label{omit} Let $G$ be a strongly connected and non-commutative group. Then $\tau_U([G,G])=\tau_U(G).$  
\end{lem}
\begin{proof} 
Implicit in the lemma is the fact that $[G,G]$ is type-definable. This follows from $\alpha$-indecomposability theorem of \cite{BerlineLascar}. In fact, if $G$ is definable then so is the commutator.
  We note that by \ref{central}, we know that If $\tau_U(H)<\tau_U(G),$ then $H \leq Z(G).$ 
 For any $a \in G,$ define 
$$c_a:G \rightarrow G$$
$$x \mapsto axa^{-1}x^{-1}.$$
But, assuming that $H \leq Z(G)$ means that $c_a$ is a  definable homomorphism from $G$ to $H.$ This is impossible because if $\tau_U(H)<\tau_U(G),$ then the kernel of the map is a subgroup of $G$ with the property that $RU(Ker c_a) \oplus RU(H) \geq RU(G)$ by the Lascar inequality. But, this means that $RU(Ker c_a) \geq \omega^\alpha n.$ This is impossible since $G$ is $\alpha$-connected. 
\end{proof}

\begin{rem} Even in the case that $[G,G]$ (or another normal abstract subgroup) is not definable, one can consider the smallest type-definable subgroup, $H,$ containing the $[G,G].$ One can still show $H$ is normal.
It appears that Cassidy and Singer \cite{CassidySinger} need this fact for their lemma 2.24, since they did not know until \cite{Jindecomposability} that commutators are definable). I will offer a proof. Take $A \lhd G$ where there are no definability conditions on $A.$ Then, let $H$ be the smallest definable subgroup containing $A$ (differentially closed fields are $\omega$-stable, so we have the descending chain condition on definable groups). Now, consider the $G$-conjugates of $H,$ if $H$ is not normal. Since $A$ is normal, each of these is still a definable subgroup containing $A$. So, $H \cap H^g$ is a definable subgroup containing $A.$ But, this contradicts the minimality of $H.$ So, $H=H^g.$ 
\end{rem}

\begin{prop} Let $G$ and $H$ be isogenous $\alpha$-connected groups. Both are almost simple or neither is.  Both are commutative or neither is.
\end{prop}
\begin{proof} We have the following diagram, since $G$ and $H$ are isogenous, 
$$\xymatrix{
G \ar[rd]^{\phi_G} & & H  \ar[ld]_{\phi_H}    \\
& K & 
}$$
$G$ commutative implies $K$ is commutative. 
Let $H_1$ be the smallest definable subgroup containing the commutator of $H.$  We know that $\tau_U(H_1)= \tau_U(H)$ by \ref{omit}. But, $\tau_U(ker \phi_H)<\tau_U(H_1),$ so the image is nontrivial. Further, it holds that the smallest definable subgroup containing the commutator of $K$ contains the image of $H_1.$ This is a contradiction, since ($K$ being commutative) the commutator is the identity.   
\end{proof}

The main reason for the notion of isogenous in this paper is to utilize it to prove uniqueness results of the form "up to isogeny" similar to the case of algebraic groups or differential algebraic groups. In particular, we will start, in the next section with a theorem similar to Baudisch's Jordan-H\"{o}lder style decomposition based on Berline-Lascar analysis of superstable groups. 

\section{Jordan-H\"{o}lder Theorems}
The proof of the following theorem follows the proof of the Jordan-H\"{o}lder theorem in the case of partial differential fields due to Cassidy and Singer. We should mention that though Lascar rank is \emph{not} the same as the notions of dimension that Cassidy and Singer use, it shares enough of the same properties to make the proofs work similarly after the correct translation of the statements is known. 
\begin{thm}
Let $G$ be an $\alpha$-connected superstable group. Then there exists a normal sequence $$1=G_r \lhd G_{r-1} \lhd \ldots \lhd G_1 \lhd G_0=G.$$
For each $i \in \{0, \ldots , r-1 \}:$ \begin{itemize}
\item $G_i$ is strongly connected and $\tau_U(G_i)=\tau_U(G).$ 
\item $RU(G_i)>RU(G_{i+1}).$
\item $G_i/G_{i+1}$ is almost simple and $\tau_U( G_i/G_{i+1})=\tau_U(G).$ $RU(G_i/G_{i+1})=\omega^\alpha \cdot (n-m),$ where $RU(G_i)=\omega^\alpha \cdot n$ and $RU(G_{i+1})=\omega^\alpha \cdot m.$  
\end{itemize}
If $$1=H_s \lhd H_{s-1} \lhd \ldots \lhd 
H_1 \lhd H_0=G$$
is another sequence which satisfies the above properties, then the sequences must be the same length ($r=s$). There is a permutation (call it $\sigma$)  of the indices so that the quotients are isogenous. That is, $G_{\sigma (i)}/ G_{\sigma (i) +1 }$ is isogenous to $H_i /H_{i+1}.$ 
\end{thm}
\begin{proof} If $G$ is already almost simple, then there is nothing to do. If this is not the case, then there is a nonempty collection of definable normal subgroups $H$ of $G$ with $\tau_U(G)=\tau_U(H).$ Pick any such $H$ so that if $RU(H)=\omega^\alpha \cdot n + \beta,$ then there is no other $H_1$ in the collection so that $RU(H)=\omega^\alpha \cdot n_1 + \beta_1,$ where $n_1>n.$ We let $G_1$ be the $\alpha$ connected component of $H, \,$ $G_1=H^{(\alpha )}.$ $G_1$ is a characteristic subgroup of $H \lhd G$, so $G_1 \lhd G.$ Since the Cantor sum on ordinals is equal to the sum when the ordinals in question are monomials,  $RU(G/G_1)=\omega^\alpha \cdot (n-m),$ where $RU(G)=\omega^\alpha \cdot n$ and $RU(G_1)=\omega^\alpha \cdot m.$ Suppose that the quotient $G/G_1$ is not almost simple. Then, there is a definable normal subgroup $H_1 \lhd G/G_1$ with $\tau_U(G)=\tau_U(G/G_1)=\tau_U(H_1).$ But, then the preimage of $H_1$ under the quotient map is a subgroup of $G$ which violates the maximality condition with which $H$ was chosen, namely, the leading monomial of the Lascar rank of the preimage of $H_1$ is larger than that of $H.$ So, the quotient is almost simple. 

From here, the proof proceeds in a similar manner to the proof of Cassidy-Singer decomposition in the differential field context. In turn, that proof follows the one in \cite{LangAlgebra} Lang Chapter 1 section 3. So, suppose we have two sequences as above $ \langle G_i \rangle _{i \leq r}$ and $\langle H_j \rangle _{j \leq s}.$ For each pair $(i,j)$ with $i<r$ and $j<s,$ we define: 
$$G_{i,j}:=G_{i+1} (H_j \cap G_i).$$
Notation: 
$$G_{i,s}:=G_{i+1,0}$$ 
Then, $$1 \lhd G_{r-1, s-1} \lhd G_{r-1, s-2} \lhd \ldots \lhd G_{r-1} \lhd G_{r-2, s-1} \lhd \ldots G_1 \lhd G_{0,s-1} \lhd \ldots G_{0,0}=G. $$ Of course, one can apply the definition in the opposite way as well, so get a refinement of $\langle H_j \rangle,$ 
$$H_{j,i}= H_{j+1} (G_i \cap H_j).$$ 
By 3.3 from Lang's algebra, $G_{i,j}/G_{i,j+1}$ is isomorphic to $H_{j,i}/H_{j,i+1}.$ Further, the isomorphism is definable. 
Claim: For and $i=0 \ldots r-1,$ there is precisely one $j$ so that $\tau_U(G_{i,j}/G_{i,j+1})= \tau_U(G).$ Further, for this specific value of $j,$ we have that $G_{i,j}/G_{i,j+1}$ is isogenous to $G_i/G_{i+1}.$ Then, since the symmetric statement holds for the $H_{j,i}$ we know that $r=s$ and the theorem follows. So, we prove the claim. By the Lascar inequality, 
$$ RU(G_{i+1}) + \sum_{j=s-1}^{0} RU(G_{i,j}/G_{i,j+1}) \leq RU(G_i) \leq RU(G_{i+1}) \oplus \bigoplus_{j=s-1}^{0} RU(G_{i,j}/G_{i,j+1})  $$
So, for some $j,$ $$\tau_U(G_{i,j}/G_{i,j+1})=\alpha.$$ Now, let $j$ be minimal so that the condition holds.  $$\tau_U(G)=\tau_U(G_{i,j}/G_{i,j+1}) \leq \tau_U(G_{i,j}/G_{i+1}) \leq \tau_U(G_i/G_{i+1}) =\tau_U(G).$$ 
Then note that for each $k<j,$ $$\tau_U(G_{i,j}/G_{i+1}) \leq \tau_U(G_{i,k}/G_{i+1} \leq \tau_U(G_i/G_{i+1}).$$
Thus, for all $K<j,$ $\tau_U(G_{i,k}/G_{i+1})=\tau_U(G).$ But, we know that $\tau_U(G_{i,k}/G_{i,k+1})<\alpha.$ However, we know that $G_i/G_{i+1}$ is $\alpha$-connected. But, this forces $G_{i,0}=G_i=G_{i,1}.$ Continuing in the same way, we can see $$G_{i,0}= \ldots = G_{i,j}.$$ We have the canonical projection map $$G_i/G_{i+1} \rightarrow G_{i,j}/ G_{i,j+1}.$$ The kernel is a proper normal subgroup of an almost simple group, so the map is an isogeny. Now, suppose that for some $t>j,$ we have that $\tau_U(G_{i,t}/G_{i,t+1})=\tau_U(G).$ Then $\tau_U(G_{i,t}/G_{i+1})=\tau_U(G).$ This contradicts the almost simplicity of $G_i/G_{i+1}.$ So, we have the desired uniqueness result. 
\end{proof}

\begin{rem}
In differential fields, there are examples, due to Cartan,  Cassidy and Singer, which show that some sort of weaker notion of correspondence than isomorphism is necessary for these sort of theorems to be true. We discuss the model theoretic aspects of an example of Cassidy and Singer here. We work in a differentially closed field with two distinguished derivations $\{\delta_1 , \delta_2 \}.$ Let $a$ be such that $\delta_1 (a)=1$ and $\delta_2 (a) =0.$ Let $G_1$ be the zero set of $(\delta_1^2 z - \delta_2z ).$ Let $G_2$ be the zero set of $(\delta_1 z-a\delta_2 z).$ We consider these groups as subgroups of the additive group. By work of Sonat Suer, \cite{Suerthesis}, both of these differential algebraic groups have Lascar rank $\omega.$ One can quickly see that $\omega$ is a lower bound for the Lascar rank, because the solution sets to the equations are infinite dimensional vector spaces over the fied of absolute constants. Seeing that $\omega$ is an upper bound takes slightly more work. We will work with a slightly more general class of examples.

Lets show $Z(\delta_1 y - f(y))$ where $f \in K[\delta_2 ]$ is almost simple. A proper subgroup of the additive group must be defined by a linear operator $g \in K[\Delta ]$ such that $g(y) \notin \{ \delta _1  y - f (y) \}$ in $K\{z\}$ and such that $H \subset \{y \in G \,  | \, g(y) = 0 \}.$ We may assume that $g \in K[\delta_2]$. If g has order $d,$ then for any $y  \in  H,$ we have $k(y,\delta_2 y,\delta_2 y,...) = k(y,\delta_2 y,...,\delta_2^{d-1} y)$. So, $H$ has $\Delta$-type 0, which implies finite Lascar rank.

$G=G_1+G_2$ is strongly connected and the series decomposition as above may be given $1 \lhd G_1 \lhd G$ or $1 \lhd G_2 \lhd G.$ Cassidy and Singer \cite{CassidySinger} show that $G_1$ is not isomorphic to either $G_2$ or $G/G_2.$ However, $G_1$ is isogenous to $G/G_2.$

\end{rem}

\begin{rem} All currently known non-commutative almost simple differential algebraic groups are actually have finite center. Such groups are, by \cite{Jindecomposability}, the perfect central extensions of the $C'$ points of an algebraic group. That is any almost simple $G$ has the following exact sequence: 
$$1 \rightarrow Z(G) \rightarrow G \rightarrow H \rightarrow 1$$ 
Now, in an arbitrary superstable theory $T$ work with an arbitrary definable perfect central extension of an algebraic group $H$, which is almost simple. Is $G$ a finite extension of $H?$ The assumptions are weak enough so that one might guess that the answer is no. However, examples which show the negative conclusion would be of interest if they could be translated to differential fields.
\end{rem}

There are suitable theories of numerical polynomials in other algebraic settings from which a theory similar to that of Cassidy and Singer might be developed. An example of model theoretic interest is the setting of difference-differential fields \cite{Medina}. In that setting, there is no notrivial lower bound Lascar rank in terms of the appropriate generalization of differential gauge (there are definable sets of Morley rank one with infinite difference-differential transcendence degree). The results in this paper would have to be generalized to the supersimple setting in order to compare potential model theoretic and algebraic notions of strong connectedness.

\bibliography{Research}{}
\bibliographystyle{plain}
\end{document}